\documentclass[12pt,a4paper]{amsart}

\usepackage[margin=2cm]{geometry}

\newcommand{\C}{\ensuremath{\mathbb{C}}}

\newcommand{\R}{\ensuremath{\mathbb{R}}}

\newcommand{\abs}[1]{\ensuremath{\lvert#1\rvert}}
\newcommand{\norm}[1]{\ensuremath{\lVert#1\rVert}}

\newcommand{\eps}{\epsilon}

\newtheorem{theorem}{Theorem}[section]
\newtheorem{lemma}[theorem]{Lemma}
\newtheorem{proposition}[theorem]{Proposition}

\theoremstyle{definition}

\DeclareMathOperator{\Tr}{Tr}

\title{Quantum unique ergodicity for random bases of spectral projections}
\author{Kenneth Maples}
\address{Institut f\"ur Mathematik, Universit\"at Z\"urich, Winterthurerstrasse 190, CH-8057 Z\"urich, Switzerland.}
\email{kenneth.maples@math.uzh.ch}

\begin{document}

\begin{abstract}
We consider a random wave model introduced by Zelditch to study the behavior of typical quasi-modes on a Riemannian manifold. Using the exponential moment method, we show that random waves satisfy the quantum unique ergodicity property with probability one under mild growth assumptions.
\end{abstract}

\maketitle

\section{Introduction}

\thispagestyle{empty}

The theory of quantum chaos is concerned with the high energy limit of quantizations of classical dynamical systems. For example, let $(M,g)$ be a compact Riemannian manifold and let $\Delta$ denote the positive Laplace-Beltrami operator. It is natural to ask how the eigenfunctions $\Delta f_k = \lambda_k f_k$ behave as the eigenvalues $\lambda_k$ grow to infinity. It is known \cite{Sch74, CdV85, Zel87} that if the classical geodesic flow $G_t$ on $S^*M$ is \emph{ergodic}, i.e.~if for every continuous function $f \in C^0(S^*M)$ we have
\[
 \lim_{T \to \infty} \frac{1}{T} \int_0^T f(G_t(x)) \, dt = \int_{S^*M} f \, d\omega
\]
then the eigenfunctions $f_k$ are \emph{quantum ergodic} in the following precise sense: For every observable (zeroeth order psuedo-differential operator) $A \in \Psi^0(M)$ there is a density $1$ subsequence $f_{k_i}$ of the eigenfunctions such that
\[
  \langle A f_{k_i}, f_{k_i} \rangle \xrightarrow{i \to \infty} \int_{S^*M} \sigma_A \, d\omega
\]
where $\sigma_A$ is the principal symbol of $A$. Thus, for classically ergodic systems the densities $\abs{f_k}^2 \, dV$ of most of the eigenfunctions will converge to the uniform measure in the high energy limit.

In the physics literature, it is expected that the high energy behavior of such quantizations of classical ergodic systems should match what is predicted by random matrix theory \cite{Wil88}. Towards this end, several models for a random sequence of functions have been proposed to simulate various limiting behaviors. In a series of articles, Zelditch \cite{Zel92, Zel96, Zel12} introduced a random matrix model to analyze the limiting behavior of a random sequence of functions which are short linear combinations of eigenfunctions, each with eigenvalue $\lambda_j$ growing to infinity. Let us now recall this model.

Fix a compact Riemannian manifold $(M,g)$ and let $P \in \Psi^1(M)$ be a first order pseudo-differential operator (see e.g.~\cite{Hor85a} for definitions). For example, as above we may choose $P = \sqrt{\Delta}$ to be the square root of the positive Laplacian. Then because $P$ is a positive Hermitian (unbounded) operator on $L^2(M)$, it has positive eigenvalues $\{\lambda_j\}_{j=1}^\infty$ where we have fixed the ordering $\lambda_j \leq \lambda_{j+1}$ for each $j \geq 1$. We can make a (non-canonical) choice of eigenfunctions $f_j \in L^2(M)$ so that $P f_j = \lambda_j f_j$ and $\langle f_j, f_k \rangle = \delta_{j,k}$.

Because $P$ is self-adjoint we can partition its spectrum into disjoint intervals $I_k \subset \R$ and construct spectral projections $\Pi_k : L^2(M) \to \mathcal{H}_k$, where $\mathcal{H}_k$ is the finite-dimensional span of the eigenfunctions with eigenvalues in $I_k$. These projections are self-adjoint.

For each spectral projection $\Pi_k$, let us write $f_1^{(k)}, \dots, f_d^{(k)}$ for the sequence of eigenfunctions with eigenvalues in $I_k$; here $d = \dim \mathcal{H}_k$. Then we can define a random orthonormal basis of $\mathcal{H}_k$ by constructing $U_k \in U(\mathcal{H}_k)$ randomly according to Haar measure and defining $g_j^{(k)} := U_k f_j^{(k)}$. Note that the joint law of the random basis $(g_j^{(k)})_{j=1}^d$ is independent of the choice of the eigenfunctions $f_j$ for repeated eigenvalues because Haar measure is invariant.

This construction can be extended to all of $L^2(M)$ in the following natural way. Let $U \in U(L^2(M))$ be the operator which acts diagonally on the block decomposition
\begin{align*}
  U : L^2(M) \cong \bigoplus_{k=1}^\infty \mathcal{H}_k &\to \bigoplus_{k=1}^\infty \mathcal{H}_k \\
  \sum_{k=1}^\infty h_k &\mapsto \sum_{k=1}^\infty U_k h_k
\end{align*}
and the sequence $(U_k)_{k=1}^\infty \in \prod_{k=1}^\infty U(\mathcal{H}_k)$ is constructed according to the product measure which projects to Haar measure on each block. In probabilistic language, we simply choose $(U_k)$ as a sequence of independent Haar unitary matrices of the appropriate dimension. We therefore have the random sequence $g_j = U f_j$ which forms a basis for $L^2(M)$.

For $\emph{any}$ sequence $\phi_j$ of functions in $L^2(M)$, we say that $(\phi_j)_{j=1}^\infty$ is \textbf{ergodic} if for all $A \in \Psi^0(M)$ with principal symbol $\sigma_A$,
\[
  \lim_{N \to \infty} \frac{1}{N} \sum_{j=1}^N \abs{\langle A \phi_j, \phi_j \rangle - \int_{S^*M} \sigma_A \, d\omega}^2 = 0.
\]
In particular, by a classical argument there is a subsequence $j_i$ with
\[
  \liminf_{N \to \infty} \frac{\#\{j_i \leq N\}}{N} = 1
\]
such that $\langle A \phi_{j_i}, \phi_{j_i} \rangle \xrightarrow{i \to \infty} \int_{S^*M} \sigma_A \, d\omega$.

Now we can recall the previous results of Zelditch on this random wave model.
\begin{theorem}[Zelditch \cite{Zel92, Zel96, Zel12}] \label{thm:zelditch}
    Suppose the spectral projections $\Pi_k$ satisfy $\dim \mathcal{H}_k \to \infty$ and the Weyl asymptotics $\frac{1}{\dim \mathcal{H}_k} \Tr \Pi_k A \Pi_k \to \int_{S^*M} \sigma_A \, d\omega$ for all $A \in \Psi^0(M)$. Then with probability one the random sequence $(g_j)_{j=1}^\infty = (U f_j)_{j=1}^\infty$ is ergodic.
\end{theorem}
Note in particular that this theorem does not require that the underlying manifold $(M,g)$ have ergodic geodesic flow. 

It is of interest to know when the dense subsequence condition above can be improved. In applications, this eliminates the possibility of high energy quantum states with ``scarring'' effects. We say that $(\phi_j)_{j=1}^\infty$ is \textbf{uniquely ergodic} if for all $A \in \Psi^0(M)$ with principal symbol $\sigma_A$,
\[
  \lim_{j \to \infty} \langle A \phi_j, \phi_j \rangle = \int_{S^*M} \sigma_A \, d\omega.
\]
If the sequence $(\phi_j)_{j=1}^\infty$ is uniquely ergodic, then it is easy to see that satisfies the ergodic property as well.

The purpose of this article is to prove that the random wave model satisfies this stronger condition as well.
\begin{theorem} \label{thm:main}
Under the same conditions as Theorem~\ref{thm:zelditch}, if additionally $\dim \mathcal{H}_k > C k^\eps$ for some $\eps > 0$ then with probability one the random sequence $(g_j)_{j=1}^\infty$ is uniquely ergodic.
\end{theorem}
The additional condition is very weak. Indeed, in most of the applications the intervals will contain approximately $k^r \log(k)^s$ for some $k, s \geq 1$.

The proof can be summarized as follows. As in the theorem of Zelditch, we show that assuming Weyl asymptotics the diagonal matrix coefficients $\langle A g_j^{(k)}, g_j^{(k)} \rangle$ can be given the explicit form
\[
  \langle A g_j^{(k)}, g_j^{(k)} \rangle = \sum_{i=1}^d \nu_i \abs{U_{ij}}^2 + \int_{S^*M} \sigma_A \, d\omega + o(1)
\]
where $(U_{ij})_{i,j=1}^d \in U(\mathcal{H}_k)$ is a random matrix distributed according to Haar measure and $\nu_i$ are the recentered eigenvalues of $\Pi_k A \Pi_k$. Then, we decompose the entries of $(U_{ij})$ with an analogue of the law of large numbers to compare the sum to a sum of independent random variables. This, in turn, can be controlled using the exponential moment method.

After the completion of this work, the author learned of results using analogous methods due to Burq and Lebeau \cite{BL11}. In their article, they show that almost every basis for the space of spherical harmonics is bounded in $L^p$ norms, as well as other estimates for specific geometric situations. This article, in contrast, proves quantum ergodicity for random bases of eigenfunctions for all geometries where the intervals of eigenvalues are growing sufficiently fast.

The organization of this article is as follows. In Section~\ref{sec:assumpt} we review the assumptions on the distribution of the eigenvalues of $P$. These assumptions will allow us to conclude that the eigenvalues of the block diagonal part of the observable $A$ are controlled as we allow the energy to grow to infinity. In Section~\ref{sec:combinatorial} we transform the matrix coefficients $\langle A g_j, g_j \rangle$ into an appropriate sum over the eigenvalues of $\Pi_k A \Pi_k$. Next, Section~\ref{sec:formulae} recalls classical formulas for the coefficients of a Haar-distributed unitary matrix.  The proof of Theorem~\ref{thm:main} is completed in Section~\ref{sec:quantum}.

\section{Assumptions} \label{sec:assumpt}

The assumptions we place on the manifold are the same as those used in \cite{Zel96, Zel12}. We first recall that the eigenvalues of $P$ have a predictable asymptotic distribution, so we can control the dimension of the spaces $\mathcal{H}_k$.

\begin{lemma}[Weyl asymptotics \cite{Hor85b}]
    Let $(M^n,g)$ be a compact Riemannian manifold. Then if $(\lambda_j)_{j=1}^\infty$ is the spectrum of $P$, then
    \[
      \#\{\lambda_j \leq \lambda\} = \operatorname{const.} \mu(S^*M) \lambda^n + O(\lambda^{n-1}).
    \]
    Furthermore, if the geodesic flow $G_t$ is aperiodic, then the error term can be improved to $o(\lambda^{n-1})$.
\end{lemma}

Next, we would like to know how the eigenvalues of the projections $\Pi_k A \Pi_k$ develop as $k \to \infty$. We can control the moments of the eigenvalues using the following lemma which is a generalization of an early result of Sz\"ego. To establish Theorem~\ref{thm:main}, we only require upper bounds on the second moment.

\begin{lemma}[Sz\"ego \cite{Gui79, Wid79}]
With $P$, $\lambda_j$ and $I_k$ as defined above, for any observable $A \in \Psi^0(M)$ and for all $m \geq 1$,
\[
  \lim_{k \to \infty} \frac{1}{\#\{\lambda_j \leq \lambda\}} \Tr (\Pi_j A \Pi_j)^m = \int_{S^*M} \sigma_A^m \, d\omega.
\]
where $\sigma_A$ is the principal symbol of $A$.
\end{lemma}

Finally, we observe that because the observable $A$ is self-adjoint, we can assume a priori absolute bounds on the eigenvalues of the projections $\Pi_k A \Pi_k$, as follows.

\begin{lemma}[Trivial bound]
Let $A \in \Psi^0(M)$ be a zero order Hermitian pseudo-differential operator on a compact Riemannian manifold $(M,g)$. Let $\Pi$ be a projection onto a $d$-dimensional subspace of $L^2(M)$, $d < \infty$, and let $-\infty < \nu_1 \leq \cdots \leq \nu_d < \infty$ denote the eigenvalues of $\Pi A \Pi$. We then have the bounds
\[
  - \norm{A} \leq \nu_1 \leq \cdots \leq \nu_d \leq \norm{A}.
\]
\end{lemma}

\begin{proof}
Note that $\norm{\Pi_k A \Pi_k} \leq \norm{A} < \infty$ uniformly in $k$, and the inequalities follow because $\Pi_k A \Pi_k$ is self-adjoint.
\end{proof}

\section{A combinatorial reduction} \label{sec:combinatorial}

We begin by converting the matrix coefficients $\langle A \psi_j^{(k)}, \psi_j^{(k)} \rangle$ into combinatorial formula in terms of the eigenvalues of $\Pi_k A \Pi_k$. In fact, we show that every matrix coefficient has a nice interpretation in terms of the coefficients of a Haar distributed unitary matrix.

\begin{proposition} \label{prop:combinatorial}
Let $P$, $A$, $I_k$, $\mathcal{H}_k$ and $\Pi_k$ be as above. Let $f_1^{(k)}, \ldots, f_d^{(k)}$ denote the eigenfunctions of $P$ with corresponding eigenvalues in $I_k$, so that they form an orthonormal basis of $\mathcal{H}_k$. Suppose $\nu_1, \ldots, \nu_d$ are the eigenvalues of $\Pi_k A \Pi_k$ and $U \in U(\mathcal{H}_k)$ is distributed according to Haar measure. Then there is a random $V \in U(\mathcal{H}_k)$, also distributed according to Haar measure, so that for all $1 \leq i, j \leq d$ we have
\[
  \langle A U f_i^{(k)}, U f_j^{(k)} \rangle = \sum_{\ell = 1}^d \nu_\ell V_{\ell, i} \overline{V_{\ell, j}}
\]
where $V_{\ell,i}$ and $V_{\ell,j}$ are the coefficients of $V$ in the basis $f_1^{(k)}, \ldots, f_d^{(k)}$. In particular, on the diagonal we have
\[
  \langle A U f_i^{(k)}, U f_i^{(k)} \rangle = \sum_{\ell = 1}^d \nu_\ell \abs{V_{\ell, i}}^2.
\]
\end{proposition}

\begin{proof}
Since $\Pi_k A \Pi_k$ is Hermitian, by the spectral theorem there exists a basis $\phi_1, \ldots, \phi_d$ of $\mathcal{H}_k$ so that $\Pi_k A \Pi_k \phi_j = \nu_j \phi_j$ for each $1 \leq j \leq d$. If we expand $U f_i^{(k)}$ and $U f_j^{(k)}$ into this basis and expand the sum, we deduce
\[
  \langle A U f_i^{(k)}, U f_j^{(k)} \rangle = \sum_{\ell = 1}^d \nu_\ell \langle U f_i^{(k)}, \phi_\ell \rangle \overline{\langle U f_j^{(k)}, \phi_\ell \rangle}.
\]
Let $T : \mathcal{H}_k \to \mathcal{H}_k$ denote the unitary change-of-basis operator mapping $f_\ell^{(k)} \mapsto \phi_\ell$. Note that this mapping is deterministic (so it does not depend on $U$). Then $\langle U f_i^{(k)}, \phi_\ell \rangle = \langle T^* U f_i^{(k)}, f_\ell^{(k)} \rangle = \langle V f_i^{(k)}, f_\ell^{(k)} \rangle$ where $V := T^* U$ is distributed according to Haar measure by invariance. We thus conclude that
\[
  \langle A U f_i^{(k)}, U f_j^{(k)} \rangle = \sum_{\ell = 1}^d \nu_\ell \langle V f_i^{(k)}, f_\ell^{(k)} \rangle \overline{\langle V f_j^{(k)}, f_\ell^{(k)} \rangle} = \sum_{\ell = 1}^d \nu_\ell V_{\ell, i} \overline{V_{\ell, j}}
\]
as required.
\end{proof}
It is convenient to recenter the eigenvalues of $\Pi_k A \Pi_k$ so that the sum is balanced around zero. This is possible precisely when the projections $\Pi_k$ satisfy local Weyl asymptotics.
\begin{proposition} \label{prop:recenter}
Let $P$, $A$, $I_k$, $\mathcal{H}_k$, $\Pi_k$, $U$, and $V$ be as above. Let $\nu_1, \ldots, \nu_d$ denote the eigenvalues of $\Pi_k A \Pi_k$. Suppose that the sequence of projections satisfies local Weyl asymptotics, i.e.
\[
  \frac{1}{d} \Tr \Pi_k A \Pi_k = \frac{1}{d} \sum_{i=1}^d \nu_i \xrightarrow{k \to \infty} \int_{S^*M} \sigma_A \, d\omega
\]
where $\sigma_A$ is the principal symbol of $A$. Then for all $1 \leq i, j \leq d$ we have
\[
  \langle A U f_i^{(k)}, U f_j^{(k)} \rangle = \delta_{i,j} \int_{S^*M} \sigma_A \, d\omega + \sum_{\ell = 1}^{d} \eta_\ell V_{\ell,i} \overline{V_{\ell,j}} + o(1)
\]
where
\[
  \eta_\ell = \nu_\ell - d^{-1} \sum_{t = 1}^d \nu_t
\]
and the implied constant is uniform in $i$ and $j$.
\end{proposition}

\begin{proof}
This follows from Proposition~\ref{prop:combinatorial} and the identities
\[
  \sum_{\ell = 1}^d V_{\ell,i} \overline{V_{\ell,j}} = \delta_{i,j}
\]
which hold because $V$ is unitary.
\end{proof}

\section{Explicit formulae for Haar unitary matrices} \label{sec:formulae}

Proposition~\ref{prop:combinatorial} suggests that it suffices for us to understand the distribution of the entries of a random unitary matrix in a fixed basis. For our application, it suffices for us to consider the distribution of the entries in one column of the matrix.  It is well-known that each column takes values uniformly in the complex unit sphere, which can easily be seen by the invariance of Haar measure. For convenience, we will state the result in probabilistic language, which amounts to a change of variables.

\begin{proposition} \label{prop:lawofsphere}
Suppose $X \in \C^d$ is distributed uniformly on the unit sphere. Then there are independent random variables $\xi_1, \ldots, \xi_d, e_1, \ldots, e_d$, with $\xi_1, \ldots, \xi_d$ uniformly distributed on $\{z \in \C \mid \abs{z} = 1\}$ and $e_1, \ldots, e_d$ with density $e^{-x} \, dx$ on $\R^+$, so that
\[
  X = (\xi_1 \sqrt{\frac{e_1}{e_1 + \cdots + e_d}}, \ldots, \xi_d \sqrt{\frac{e_d}{e_1 + \cdots + e_d}})
\]
as random variables.
\end{proposition}

\begin{proof}
Recall that the uniform distribution on the unit sphere in $\C^d$ can be constructed with a random vector $Y \in \R^{2d}$ with Gaussian density $e^{- \pi \abs{y}^2} \, dy$. Indeed, if $Y$ is such a vector then with probability one it is non-zero, so the normalized vector $Y / \abs{Y}$ is well-defined and on the unit circle in $\R^{2d} \cong \C^d$ with the mapping $(y_1, \ldots, y_{2d}) \mapsto (y_1 + i y_2, \ldots, y_{2d-1} + i y_{2d})$. The Gaussian is trivially invariant under orthogonal transformation so the density of $Y / \abs{Y}$ is uniform on the unit sphere.

To compute the densities, for each $1 \leq k \leq d$ we change variables to see
\[
  e^{- \pi (y_{2k-1}^2 + y_{2k}^2)} \, dy_{2k-1} dy_{2k} = r_k \exp(-\pi r_k^2) \, dr d\xi_k = \exp(- \pi e_k) \, de_k d\xi_k.
\]
Here each $\xi_k$ is uniform on the unit circle in $\C$. Thus, the result follows after scaling the densities of $e_1, \ldots, e_d$ by a common factor so that their density on $\R^+$ is $e^{-x} \, dx$, as required.
\end{proof}

The denominators in the previous formula are inconvenient because they introduce dependencies between the different coefficients of the matrix. The next proposition shows that we can replace them with a constant factor at the cost of an arbitrarily small error in our representation, so that the coefficients are close to independent.

\begin{lemma}[Law of large numbers] \label{lem:sumofexponential}
Let $e_1, \ldots, e_d$ be iid exponential random variables with distributions $e^{-x} \, dx$. Then for all $0 < \delta$ less than some absolute constant, there is a random variable $\theta \in \R$ with $\abs{\theta} < \delta$ almost surely and
\[
  e_1 + \cdots + e_d = (1 + \theta) d
\]
with probability $1 - O(e^{- C \delta^2 d})$, where the constants are absolute.
\end{lemma}

\begin{proof}
We begin with the upper bound. We bound by the exponential moment and the independence of $e_1, \ldots, e_d$, for all $0 < t < 1$,
\[
  \mathbb{P}(e_1 + \cdots + e_d > (1 + \delta) d) \leq \exp(-t(1 + \delta) d) \prod_{j=1}^d \mathbb{E} \exp(t e_j).
\]
We have
\[
  \mathbb{E} \exp(t e_j) = \int_0^\infty \exp((t-1) x) \, dx = \frac{1}{1 - t}
\]
by computation. With the asymptotic
\[
  \exp(-t-\frac12 t^2) \geq 1 -t \geq \exp(-t-\frac{5}{8} t^2)
\]
for $0 < t < 1/4$, we bound
\begin{align*}
  \mathbb{P}(e_1 + \cdots + e_d > (1 + \delta) d) &\leq \exp(-t(1 + \delta) d + td + \frac58 t^2 d) \\
  &\leq \exp(-t\delta d + \frac{5}{8} t^2 d).
\end{align*}
On the interval $0 < t < 1/4$ where this equation is valid, the coefficient of $d$ is minimized at $t = \frac{4}{5} \delta$, assuming $\delta < \frac{5}{16}$. We thus have
\[
  \mathbb{P}(e_1 + \cdots + e_d > (1 + \delta) d) \leq \exp(- \frac45 \delta^2 d)
\]
as required.

For the lower bound,
\[
  \mathbb{P}(e_1 + \cdots + e_d < (1 - \delta) d) \leq \exp(t (1 - \delta) d) \prod_{j=1}^d \mathbb{E} \exp(-t e_j)
\]
for all $0 < t$. As before we have
\[
  \mathbb{E} \exp(-t e_j) = \frac{1}{1 + t} \leq \exp(-t + \frac12 t^2)
\]
for all positive $t$, so
\[
  \mathbb{P}(e_1 + \cdots + e_d < (1 - \delta) d) \leq \exp(-\delta t d + \frac12 t^2 d)
\]
and, choosing $t = \delta$ the lower bound follows.

To construct $\theta$, we thus let $E$ denote the event that both the upper and lower bounds hold and set
\[
  \theta = (d^{-1}(e_1 + \cdots + e_d) - 1) 1_E
\]
which verifies the conditions required.
\end{proof}

\section{Quantum ergodicity via decorrelation} \label{sec:quantum}

In this section we show how to use the decomposition of a random matrix to analyze the random sums appearing in Proposition~\ref{prop:combinatorial}.

In \cite{Zel12}, Zelditch controls the sums over eigenvalues from Proposition~\ref{prop:combinatorial} by computing the moments of a certain polytope. Namely, if $\nu_1, \ldots, \nu_d$ are the eigenvalues of $\Pi_k A \Pi_k$, then he shows that it suffices to compute the moments of the polytope $P$ which is the set of convex combinations of $\nu_{\tau(1)}, \ldots, \nu_{\tau(d)}$ for $\tau \in S_d$. It turns out that we can control the matrix coefficients simply with Proposition~\ref{lem:sumofexponential} and the exponential moment method. The following argument is standard in combinatorial probability.
\begin{proposition} \label{prop:expmomentmethod}
Let $\nu_1, \ldots, \nu_d \in \R$ be such that
\[
  -D \leq \nu_1 \leq \cdots \leq \nu_d \leq D
\]
and so that $\sum_{k=1}^d \nu_k = 0$. Furthermore, suppose that
\[
  M := \sum_{k=1}^d \abs{\nu_k}^2 < \infty.
\]
Let $U \in U(d)$ be a Haar distributed random unitary matrix. Then for each $1 \leq i \leq d$ we have the large deviation inequality
\[
  \mathbb{P}(\abs{\sum_{k=1}^d \nu_k \abs{U_{k,i}}^2} > \alpha) = O(\exp(-c'' \alpha^2 d^2 / M) + M \exp(-c d))
\]
which holds for all $\alpha = O(M D^{-1} d^{-1})$, where the constants are absolute.
\end{proposition}

\begin{proof}
By Proposition~\ref{prop:lawofsphere} and Lemma~\ref{lem:sumofexponential} we can write the sum in the form
\[
  \sum_{k=1}^d \nu_k \abs{U_{k,i}}^2 = \sum_{k=1}^d \nu_k \frac{e_k}{e_1 + \cdots + e_d}
\]
where $e_1, \ldots, e_d$ are a family of independent exponential random variables with density $e^{-x} \, dx$ on $\R^+$. Fix a $\delta > 0$. By Lemma~\ref{lem:sumofexponential}, there is an event $E$ with $\mathbb{P}(E) = 1 - e^{-c \delta^2 d}$ and a random variable $\theta$ with $\abs{\theta} < \delta$ such that
\[
\sum_{k=1}^d \nu_k \frac{e_k}{e_1 + \cdots + e_d} = 1_E \frac{1}{d(1 + \theta)} \sum_{k=1}^d \nu_k e_k + O(1_{E^c} M),
\]

where we applied Cauchy-Schwarz in the second term. We have by Markov's inequality for any $t > 0$ and $\alpha > 0$
\[
  \mathbb{P}(\abs{\sum_{k=1}^d \nu_k e_k} > \alpha d) \leq 2 \exp(-t \alpha d) \mathbb{E} \exp(t \sum_{k=1}^d \lambda_k e_k)
\]
Because the exponential random variables are independent,
\[
  \mathbb{E} \exp(t \sum_{k=1}^d \nu_k e_i) = \prod_{k=1}^d \mathbb{E} \exp(t \nu_k e_k) = \prod_{k=1}^d \int_0^\infty \exp(t \nu_k s - s) \, ds.
\]
Let us suppose that $2 t D \leq 1$. Then the integral converges and we have
\[
  \mathbb{E} \exp(t \sum_{k=1}^d \nu_k e_k) = \prod_{k=1}^d \frac{1}{1 - t \nu_k} = \exp(\sum_{k=1}^d \nu_k t + O((\nu_k t)^2)).
\]
Combining terms and using the fact that $\nu_1 + \cdots + \nu_d = 0$, we get
\[
  \mathbb{E} \exp(t \sum_{k=1}^d \nu_k e_k) = \exp(O(\sum_{k=1}^d (\nu_k t)^2)) = \exp(O(Mt^2)).
\]
We conclude that
\[
  \mathbb{P}(\sum_{k=1}^d \nu_k \abs{U_{k,i}}^2 > \alpha) \leq 2 \exp(-t\alpha d + O(M t^2))
\]
for all $2 t D \leq 1$. Set $t = c' \alpha d M^{-1}$ for $c' > 0$ such that the right hand side is bounded by $\exp(-c'' \alpha^2 d^2 / M)$
\end{proof}

\begin{proof}[Proof of Theorem~\ref{thm:main}]
Let $E_{k,\alpha}$ denote the event that $\abs{\langle A g_j^{(k)}, g_j^{(k)}\rangle - \int_{S^*M} \sigma_A \, d\omega} > \alpha$ for some $1 \leq j \leq \dim \mathcal{H}_k =: d$. By Proposition~\ref{prop:combinatorial} and Proposition~\ref{prop:recenter}, this is bounded by the event that $\abs{\sum_{k=1}^d \nu_k \abs{U_{k,j}}^2} > \alpha$ for some $1 \leq j \leq d$, where $\nu_1 + \cdots + \nu_d = 0$. By the union bound and Proposition~\ref{prop:expmomentmethod} we conclude that
\[
  \mathbb{P}(E_{k,\alpha}) \leq O(d \exp(-c'' \alpha^2 d^2 / M) + M \exp(-c d)).
\]
By the Sz\"ego asymptotics (or just our assumption on the upper bound of the second moment) we know that $M = \sum_{j=1}^d \nu_j \to M_\infty$ as $k \to \infty$. If we let $\alpha = C d^{-1} (\log d)$ for some $C > 0$ suitably large then the right hand side is bounded by $O(d^{-C})$. In particular,
\[
  \sum_{k=1}^\infty \mathbb{P}(E_{k,\alpha}) \leq \sum_{k=1}^\infty (\dim \mathcal{H}_k)^{-C} < \infty
\]
By the Borel-Cantelli lemma and the assumptions on $\dim \mathcal{H}_k$, with probability one at most a finite number of the events $E_{k,\alpha}$ can hold, and the result follows as $\alpha \to 0$.
\end{proof}

\section{Acknowledgements}
The author would like to thank Emmanuel Kowalski for directing him to the work leading to present article and Steve Zelditch for helpful comments.

\bibliographystyle{abbrv}
\bibliography{que.wave}

\begin{thebibliography}{10}

\bibitem{BL11}
N.~Burq and G.~Lebeau.
\newblock Injections de sobolev probabilistes et applications.
\newblock {\tt arXiv:1111.2069v1 [math.AP]}, November 2011.

\bibitem{CdV85}
Y.~Colin~de Verdi{\`e}re.
\newblock Ergodicit\'e et fonctions propres du laplacien.
\newblock {\em Comm. Math. Phys.}, 102(3):497--502, 1985.

\bibitem{Gui79}
V.~Guillemin.
\newblock Some classical theorems in spectral theory revisited.
\newblock In L.~H\"ormander, editor, {\em Seminar on Singularities of Solutions
  of Linear Partial Differential Equations}, volume~91, pages 219--259,
  Princeton, 1979. Ann. of Math. Stud., Princeton Univ. Press.

\bibitem{Hor85a}
L.~H{\"o}rmander.
\newblock {\em The analysis of linear partial differential operators. {III}},
  volume 274 of {\em Grundlehren der Mathematischen Wissenschaften [Fundamental
  Principles of Mathematical Sciences]}.
\newblock Springer-Verlag, Berlin, 1985.
\newblock Pseudodifferential operators.

\bibitem{Hor85b}
L.~H{\"o}rmander.
\newblock {\em The analysis of linear partial differential operators. {IV}},
  volume 275 of {\em Grundlehren der Mathematischen Wissenschaften [Fundamental
  Principles of Mathematical Sciences]}.
\newblock Springer-Verlag, Berlin, 1985.

\bibitem{Sch74}
A.~I. {\v{S}}nirel'man.
\newblock Ergodic properties of eigenfunctions.
\newblock {\em Uspehi Mat. Nauk}, 29(6(180)):181--182, 1974.

\bibitem{Wid79}
H.~Widom.
\newblock Eigenvalue distribution theorems for certain homogeneous spaces.
\newblock {\em J. Funct. Anal.}, 32:139--147, 1979.

\bibitem{Wil88}
M.~Wilkinson.
\newblock Random matrix theory in semiclassical quantum mechanics of chaotic
  systems.
\newblock {\em J. Phys. A}, 21(5):1173--1190, 1988.

\bibitem{Zel87}
S.~Zelditch.
\newblock Uniform distribution of eigenfunctions on compact hyperbolic
  surfaces.
\newblock {\em Duke Math. J.}, 55(4):919--941, 1987.

\bibitem{Zel92}
S.~Zelditch.
\newblock Quantum ergodicity on the sphere.
\newblock {\em Comm. Math. Phys.}, 146(1):61--71, 1992.

\bibitem{Zel96}
S.~Zelditch.
\newblock A random matrix model for quantum mixing.
\newblock {\em Internat. Math. Res. Notices}, (3):115--137, 1996.

\bibitem{Zel12}
S.~Zelditch.
\newblock Random orthonormal bases of spaces of high dimension.
\newblock {\tt arXiv:1210.2069 [math.SP]}, October 2012.

\end{thebibliography}

\end{document}